\documentclass{article}

\usepackage{xspace} 

\usepackage{amsmath}
\usepackage{amsfonts}
\usepackage{amssymb}
\usepackage{dsfont} 
\allowdisplaybreaks 

\usepackage{amsthm} 
\usepackage[dvipsnames]{xcolor} 

\usepackage{hyperref}
\hypersetup{
    colorlinks,
    linkcolor={NavyBlue},
    citecolor={NavyBlue},
    urlcolor={black}}

\usepackage{graphicx}
\usepackage{caption}
\usepackage{subcaption}
\usepackage{tikz}

\usepackage{geometry}
\renewcommand\labelenumi{(\roman{enumi})}
\renewcommand\theenumi\labelenumi

\makeatletter
\newcommand\footnoteref[1]{\protected@xdef\@thefnmark{\ref{#1}}\@footnotemark}
\makeatother

\newtheorem{theorem}{Theorem}[section]
\newtheorem{proposition}[theorem]{Proposition}
\newtheorem{lemma}[theorem]{Lemma}
\newtheorem{corollary}[theorem]{Corollary}
\newtheorem{remark}[theorem]{Remark}

\theoremstyle{definition}

\newtheorem*{acks}{Acknowledgments}

\newcommand{\E}{\mathbb{E}}

\renewcommand{\P}{\mathbb{P}}

\newcommand{\R}{\mathbb{R}}

\newcommand{\cN}{\mathcal{N}}

\newcommand{\ie}{i.e.\@\xspace}
\newcommand{\eg}{e.g.\@\xspace}

\renewcommand{\tilde}{\widetilde}

\renewcommand{\d}[1]{\mathop{}\!\mathrm{d}#1}
\newcommand{\e}{\mathrm{e}}

\let\originalleft\left
\let\originalright\right
\renewcommand{\left}{\mathopen{}\mathclose\bgroup\originalleft}
\renewcommand{\right}{\aftergroup\egroup\originalright}

\title{Upper moderate deviation probabilities for the maximum of branching Brownian motion}
\author{Louis \textsc{Chataignier}\footnote{Institut de Mathématiques de Toulouse, UMR 5219, Université de Toulouse, CNRS, UPS, F-31062 Toulouse Cedex 9, France.}}
\date{\today}

\begin{document}

\maketitle

\begin{abstract}
    It is known from~\cite{Bramson1983} that the maximum of branching Brownian motion at time $t$ is asymptotically around an explicit function $m_t$, which involves a first ballistic order and a logarithmic correction.
    In this paper, we give an asymptotic equivalent for its upper moderate deviation probability, that is, the probability that the maximum achieves $m_t + x_t$ at time $t$, where $1 \ll x_t \ll t$.
    We adopt a probabilistic approach that employs a modified version of the second moment method.
    As a byproduct, we obtain information about the typical behavior of particles contributing to such deviations.
\end{abstract}

\section{Introduction}

The (binary) branching Brownian motion (BBM) is a continuous-time branching Markov process constructed as follows.
At time $t = 0$, a single particle starts a standard Brownian motion in $\R$ from the origin.
After a random time that follows an exponential distribution with parameter $1$, it splits into two, or equivalently, it dies and gives birth to two children.
These new particles then repeat the same process, independently of each other.
More precisely, the children perform independent Brownian motions from the position of their parent at its death and, after independent exponential times, they in turn split into two.
For a formal construction, see~\cite{Chauvin1991,Chataignier2024}.
BBM has been much studied since a deep connection with the F--KPP equation was made in~\cite{ItoMcKean1965, IkedaNagasawaWatanabe1968, McKean1975}.

Let $\cN_t$ denote the set of particles alive at time $t$ and $X_t(u)$ the position of particle $u$ at time $t$ or of its ancestor alive at time $t$.
We are interested in the maximal displacement $M_t = \max_{u \in \cN_t} X_t(u)$.
In~\cite{Bramson1983}, using that $u(t, x) = \P(M_t > x)$ is solution to the F--KPP equation, Bramson gave a precise estimate for $M_t$.
More specifically, let us define $m_t = \sqrt{2}t - \frac{3}{2\sqrt{2}} (\log t)_+$, where $y_+ = \max(y, 0)$ is the positive part of $y$.
Bramson showed that $M_t - m_t$ converges in distribution to some non-degenerate random variable as time goes to infinity.
The limit distribution was subsequently identified by Lalley and Sellke~\cite{LalleySellke1987}.
Then, in~\cite{ChauvinRouault1988}, Chauvin and Rouault obtained an asymptotic equivalent for the upper large deviation probability, \ie the probability that $M_t > m_t + xt$, with $x > 0$ fixed.
In~\cite{ChauvinRouault1990}, the same authors obtained an asymptotic equivalent for the probability that $M_t > \sqrt{2}t$.
In both works, they used the F--KPP equation.
Here, we extend the latter result by studying an intermediate regime between typical fluctuations and upper large deviations, called \emph{upper moderate deviation}.
It concerns the event where $M_t > m_t + x_t$, with $x_t$ any quantity that goes to infinity with time while being negligible compared to $t$.
Our main result Theorem~\ref{th:moderate_deviations} provides an asymptotic equivalent for the probability of this event.

Throughout this paper, the letters $t$, $s$, and $r$ will refer to some points in time $[0, \infty)$.
Given two functions $f \geq 0$ and $g > 0$, we will write $f(t) \ll g(t)$ if $\lim f(t)/g(t) = 0$, and $f(t) \sim g(t)$ if $\lim f(t)/g(t) = 1$.
It will also be convenient to write $f(t) \lesssim g(t)$ if there exists a constant $C$ such that $f(t) \leq C g(t)$ for any $t$.
Let us clarify that if $f$ and $g$ depend on an additional parameter $x$, then a statement of the form ``$\forall x, \forall t, f(t, x) \lesssim g(t, x)$'' shall be understood as ``$\exists C > 0, \forall x, \forall t, f(t, x) \leq C g(t, x)$''.
Given $x \in \R$ and $A \subseteq \R$, we will write $x \pm A = \{x \pm a : a \in A\}$.

Here is the main result of this article, which directly follows from Propositions~\ref{prop:expectation_equivalent} and~\ref{prop:equivalent_expectation} below.

\begin{theorem}\label{th:moderate_deviations}
    For any $(x_t)_{t \geq 0}$ satisfying $1 \ll x_t \ll t$ as $t \to \infty$, we have the asymptotic equivalent $\P(M_t > m_t + x_t) \sim C^* \gamma_t(x_t)$ as $t \to \infty$, where, for $x > 0$,
    \begin{equation}\label{eq:definition_gamma}
        \gamma_t(x) = x \e^{-\sqrt{2}x - \frac{x^2}{2t} + \frac{3}{2\sqrt{2}} \frac{x \log t}{t}} \text{ and } C^* = \lim_{\ell \to \infty} \sqrt{\frac{2}{\pi}} \int_0^\infty y \e^{\sqrt{2}y} \P \left( M_\ell > \sqrt{2}\ell + y \right) \d{y}.
    \end{equation}
\end{theorem}

Since the theorem is stated for general choices of $(x_t)_{t \geq 0}$, an argument by contradiction shows that the equivalent holds uniformly for $x_t$, in the following sense.

\begin{corollary}\label{cor:moderate_deviation}
    For any $(a_t)_{t \geq 0}$ and $(b_t)_{t \geq 0}$ such that $1 \ll a_t < b_t \ll t$ as $t \to \infty$,
    \begin{equation*}
        \lim_{t \to \infty} \sup_{x \in [a_t, b_t]} \left| \frac{\P(M_t > m_t + x)}{C^* \gamma_t(x)} - 1 \right| = 0.
    \end{equation*}
\end{corollary}

As observed by Bovier and Hartung in~\cite{BovierHartung2020}, it is possible to adapt the proofs of~\cite[Proposition~4.4]{ArguinBovierKistler2013} and~\cite[Proposition~3.1]{BovierHartung2014} to show Theorem~\ref{th:moderate_deviations}.
Such an argument relies on a sharp estimate for solutions to the F--KPP equation.
One can find this estimate in~\cite[Proposition~4.3]{ArguinBovierKistler2013}, which is itself based on~\cite[Proposition~8.3]{Bramson1983}.
Mind that \cite[Proposition~2.1]{BovierHartung2020} is, however, incorrect. Indeed, it is stated for the regime $0 < x_t \ll t$, whereas one must restrict to $1 \ll x_t \ll t$ (see \eg~\cite{LalleySellke1987} for $x_t$ constant).
Moreover, the statement is missing the term in $(x_t/t)(\log t)$ in the exponential, which is essential unless one assumes $x_t \ll t/\log t$.

In this work, we propose a probabilistic approach that employs a modified version of the second moment method.
We adapt arguments from Bramson, Ding, and Zeitouni~\cite{BramsonDingZeitouni2016}, where the authors studied the asymptotic behavior of the maximum of branching random walk.
A first benefit of this approach is to give some understanding of the process when it performs upper moderate deviation (see Remark~\ref{rem:byproduct}), especially concerning the trajectories of the particles.
Another one is to establish a method that applies to other models, such as branching random walk.
Finally, our proof is fairly self-contained: it only requires the many-to-one and many-to-two lemmas as well as some Brownian estimates.

\begin{remark}\label{rem:byproduct}
    Here we summarize what can be inferred, as a byproduct of our analysis, about the typical behavior of a particle performing moderate deviation.
    Proposition~\ref{prop:upper_bound_G} together with Lemma~\ref{lem:lower_bound} indicates that, with high probability, such a particle stays below the line $g(s) = (m_t/t)s+x_t$ until time $s = t-\ell$, where $\ell$ should be thought of as large but constant in $t$ (in the proofs, we take $\ell = \ell_t$ going to infinity arbitrarily slowly as $t \to \infty$).
    What is more, Lemma~\ref{lem:equivalent_Lambda_Lambda2} implies that particles performing moderate deviation at time $t$ all come from the same ancestor at time $t-\ell$ (see Remark~\ref{rem:equivalent_Lambda_Lambda2}).
    Finally, Lemma~\ref{lem:equivalent_expectation} tells that this ancestor is at a distance of order $\sqrt{\ell}$ from $g(t-\ell)$ at time $t-\ell$.
    This is not surprising since a Brownian bridge of length $t$ conditioned to stay below a line for most of its lifespan, actually stays at least at a distance of order $\min(\sqrt{s},\sqrt{t-s})$ below this line.
    This phenomenon of \emph{entropic repulsion} was already present in Bramson's work~\cite{Bramson1978} and is precisely formulated in~\cite[Theorem~2.3]{ArguinBovierKistler2011}.
\end{remark}

\section{Many-to-few lemmas}

From now on, we consider $(B_t)_{t \geq 0}$ a standard Brownian motion independent of the BBM.
The following many-to-one lemma is a fundamental tool to compute first moments for certain functionals of the trajectories.

\begin{lemma}[Many-to-one]\label{lem:many-to-one}
    For any $t \geq 0$ and any non-negative measurable function $F$,
    \begin{equation*}
        \E \left[ \sum_{u \in \cN_t} F((X_s(u))_{s \leq t}) \right] = \e^t \E \left[ F((B_s)_{s \leq t}) \right].
    \end{equation*}
\end{lemma}

We will also need a many-to-two lemma to compute some second moments.
Here, we state a generalization of~\cite[Lemma~10]{Bramson1978}.

\begin{lemma}[Many-to-two]\label{lem:many-to-two}
    For any $t \geq 0$ and any non-negative measurable function $F$,
    \begin{equation*}
        \E \left[ \sum_{u, v \in \cN_t, u \neq v} F((X_s(u))_{s \leq t}, (X_s(v))_{s \leq t}) \right] = 2 \int_0^t \e^{2t-s} \E \left[ F((B_r^{1,s})_{r \leq t}, (B_r^{2,s})_{r \leq t}) \right] \d{s},
    \end{equation*}
    where $(B_r^{1,s})_{r \geq 0}$ and $(B_r^{2,s})_{r \geq 0}$ are two Brownian motions that coincide until time $s$ and have independent increments afterward.
\end{lemma}

See \eg~\cite{Chataignier2024} for proofs of Lemmas~\ref{lem:many-to-one} and~\ref{lem:many-to-two}.

\section{Brownian estimates}

Classically, for any $x_1, x_2 > 0$ and any $t > 0$, we have the following measure equality on the domain $\{y \in \R : y < x_2\}$,
\begin{equation}\label{eq:brownian_estimate_0}
    \P \left( \forall s \leq t, B_s \leq \frac{t-s}{t}x_1 + \frac{s}{t}x_2, B_t \in \d{y} \right) = \left( 1 - \e^{-2x_1(x_2-y)/t} \right) \frac{\e^{-y^2/2t}}{\sqrt{2 \pi t}} \d{y}.
\end{equation}
This is a simple consequence of~\cite[Lemma~2]{Bramson1978} and of the fact that a Brownian motion is the sum of~$(\frac{s}{t}B_t)_{s \leq t}$ and an independent Brownian bridge.
Note that, on the complementary domain $\{y \in \R : y \geq x_2\}$, the measure on the left-hand side is just zero.
From Equation~\ref{eq:brownian_estimate_0}, we deduce the following upper and lower bounds.

\begin{lemma}\label{lem:brownian_estimate_1}
    \begin{enumerate}
        \item For any $x_1, x_2 > 0$ and any $t > 0$,
        \begin{equation*}
            \P \left( \forall s \leq t, B_s \leq \frac{t-s}{t}x_1 + \frac{s}{t}x_2, B_t \in \d{y} \right) \lesssim \frac{x_1(x_2-y)}{t^{3/2}} \e^{-y^2/2t} \d{y},
        \end{equation*}
        on the domain $\{y \in \R : y < x_2\}$.
        \item For any $x_1, x_2 > 0$ and any $t > 0$,
        \begin{equation*}
            \P \left( \forall s \leq t, B_s \leq \frac{t-s}{t}x_1 + \frac{s}{t}x_2, B_t \in \d{y} \right) \gtrsim \frac{x_1(x_2-y)}{t^{3/2}} \e^{-y^2/2t} \d{y},
        \end{equation*}
        on the domain $\{y \in \R : y < x_2 \text{ and } x_1(x_2-y) \leq t\}$.
    \end{enumerate}
\end{lemma}

In the next section, we will also need to deal with a slight modification of the barrier.

\begin{lemma}\label{lem:brownian_estimate_2}
    Fix any $\alpha \in (0, 1/2)$ and $K > 0$.
    Define $\delta_t(s) = K \min(s^\alpha, (t-s)^\alpha)$.
    Then there exist $C = C(\alpha, K) > 0$ and $t_0 = t_0(\alpha, K) > 0$ such that, for any $x_1, x_2 > 0$ and any $t \geq t_0$,
    \begin{equation*}
        \P \left( \forall s \leq t, B_s \leq \frac{t-s}{t}x_1 + \frac{s}{t}x_2 + \delta_t(s), B_t \in \d{y} \right) \leq C \frac{(x_1+1)(x_2-y+1)}{t^{3/2}} \e^{-y^2/2t} \d{y},
    \end{equation*}
    on the domain $\{y \in \R : y < x_2\}$.
\end{lemma}

In the case $x_1(x_2 - y) \leq t$, Lemma~\ref{lem:brownian_estimate_2} is a direct consequence of~\cite[Lemma~2.7]{KimLubetzkyZeitouni2023}.
It suffices to drop the barrier to get the case $x_1(x_2 - y) > t$.

\section{A first upper bound}

Recall that $\gamma_t(x)$ is defined in~\eqref{eq:definition_gamma}.
We begin with the following upper bound.

\begin{lemma}\label{lem:upper_bound}
    For any $x \geq 1$ and any $t \geq 2$, we have $\P(M_t > m_t + x) \lesssim \gamma_t(x)$.
\end{lemma}

Lemma~\ref{lem:upper_bound} is stated in~\cite[Corollary~10]{ArguinBovierKistler2012} whose proof relies on F--KPP estimates from~\cite{Bramson1983}.
Here, we choose to present a probabilistic proof inspired by~\cite{Mallein2015}, that treated the case $x \in [1, \sqrt{t}]$.

\begin{proof}
    If $x \geq t$, then, by Lemma~\ref{lem:many-to-one},
    \begin{equation*}
        \P(M_t > m_t + x) \leq \e^t \P(B_t > m_t + x) \leq \e^t \frac{\sqrt{t}}{\sqrt{2 \pi} x} \e^{-(m_t + x)^2/2t} \lesssim \gamma_t(x).
    \end{equation*}
    Now, assume $x < t$.
    Consider the barriers $f_t(s) = m_t - m_{t-s} + x$ and $\tilde{f}_t(s) = f_t(s) - \sqrt{2}s$.
    To simplify notations, assume that $t$ is an integer.
    Applying successively the union bound and Lemma~\ref{lem:many-to-one}, we obtain
    \begin{align}
        \P(M_t > m_t + x) &\leq \sum_{k=0}^{t-1} \E \left[ \sum_{u \in \cN_{k+1}} \mathds{1}_{\{\forall s \leq k, X_s(u) \leq f_t(s)\}} \mathds{1}_{\{\exists r \in [k, k+1], X_r(u) > f_t(r)\}} \right] \nonumber \\
        &= \sum_{k=0}^{t-1} \e^{k+1} \P(\forall s \leq k, B_s \leq f_t(s), \exists r \in [k, k+1], B_r > f_t(r)). \label{eq:upper_bound_barrier}
    \end{align}
    Let us show that, for each $k = 0, 1, \ldots, t-1$,
    \begin{equation}\label{eq:upper_bound_k}
        \P(\forall s \leq k, B_s \leq f_t(s), \exists r \in [k, k+1], B_r > f_t(r)) \lesssim \frac{\gamma_t(x) \e^{-k} t^{3/2}}{(t-k)^{3/2}(k+1)^{3/2}}.
    \end{equation}
    If $k=0$, then the left-hand side of~\eqref{eq:upper_bound_k} is bounded by $\P \left( \sup_{r \leq 1} B_r > x \right) \lesssim \e^{-x^2/2} \lesssim \gamma_t(x)$.
    Now, assume $1 \leq k \leq t-1$.
    Since the function~$s \mapsto f_t(s)$ is increasing, the left-hand side of~\eqref{eq:upper_bound_k} is bounded by
    \begin{equation*}
        \P \left( \forall s \leq k, B_s \leq f_t(s), \sup_{r \in [k, k+1]} (B_r - B_k) > f_t(k) - B_k \right).
    \end{equation*}
    We then decompose the above probability over the values of $B_k$, we apply the Markov property, and we use that the supremum of a Brownian motion between $0$ and $1$ has same law as $|B_1|$.
    This allows us to bound the left-hand side of~\eqref{eq:upper_bound_k} by
    \begin{equation}\label{eq:upper_bound_decomposition}
        \int_0^\infty \P(\forall s \leq k, B_s \leq f_t(s), B_k \in f_t(k) - \d{y}) \P(|B_1| > y).
    \end{equation}
    Now, by Girsanov's theorem, for any non-negative measurable function $F$,
    \begin{equation}\label{eq:girsanov_transformation}
        \E \left[ F((B_s)_{s \leq k}) \right] = \E \left[ \e^{-\sqrt{2}B_k-k} F((B_s+\sqrt{2}s)_{s \leq k}) \right].
    \end{equation}
    In particular, we can rewrite the first factor in~\eqref{eq:upper_bound_decomposition} as
    \begin{align}
        \E & \left[ \e^{-\sqrt{2}B_k - k} \mathds{1}_{\{ \forall s \leq k, B_s \leq \tilde{f}_t(s), B_k \in \tilde{f}_t(k) - \d{y} \}} \right] \nonumber \\
        &= \frac{t^{3/2}}{(t-k)^{3/2}} \e^{-\sqrt{2}x + \sqrt{2}y - k} \P \left( \forall s \leq k, B_s \leq \tilde{f}_t(s), B_k \in \tilde{f}_t(k) - \d{y} \right) \nonumber \\
        &\lesssim \frac{t^{3/2}}{(t-k)^{3/2}} \e^{-\sqrt{2}x + \sqrt{2}y - k} \frac{(x+1)(y+1)}{k^{3/2}} \e^{-\left( \tilde{f}_t(k) - y \right)^2/2k} \d{y}, \label{eq:upper_bound_Girsnov_and_ballot}
    \end{align}
    by Lemma~\ref{lem:brownian_estimate_2}.
    Then we expand the exponent, we rewrite $-\frac{x^2}{2k} = -\frac{x^2}{2t} - \frac{x^2(t-k)}{2kt}$, and we use that $\frac{1}{k} \log \frac{t}{t-k} \leq \frac{1}{t} \log t$.
    Doing so, we obtain
    \begin{equation*}
        \eqref{eq:upper_bound_decomposition} \lesssim \frac{t^{3/2}}{(t-k)^{3/2}k^{3/2}} \gamma_t(x) \e^{-k} \left( \e^{-\frac{x^2(t-k)}{2kt}} \int_0^\infty (y+1) \e^{\sqrt{2}y + xy/k - y^2/2k - y^2/2} \d{y} \right).
    \end{equation*}
    By completing the square in the integral and then applying some algebraic manipulations, one can bound the factor in brackets, up to a multiplicative constant, by
    \begin{equation}\label{eq:upper_bound_bounded_integral}
        \e^{-\frac{x^2(t-k)}{2kt}} \left( \frac{x}{k} + 1 \right) \e^{\frac{k}{2(k+1)} \left(  \sqrt{2} + \frac{x}{k} \right)^2} \lesssim \left( \frac{x}{k} + 1 \right) \e^{-\frac{x^2}{2k^2} \left( \frac{k^2}{k+1} - \frac{k^2}{t} \right) + \sqrt{2}\frac{x}{k}}.
    \end{equation}
    If $1 \leq k \leq \frac{t}{4}$, then $\frac{k^2}{k+1} - \frac{k^2}{t} \geq \frac{k}{2} - \frac{k}{4} \geq \frac{1}{4}$ and $\eqref{eq:upper_bound_bounded_integral} \lesssim h(x/k)$, where $h : z \mapsto (z + 1) \e^{-\frac{z^2}{8} + \sqrt{2}z}$ is a bounded function.
    If $k > \frac{t}{4}$, then $\frac{k^2}{k+1} - \frac{k^2}{t} \geq 0$ and $\frac{x}{k} < 4\frac{x}{t} < 4$, so we still have a constant bound.
    In both cases,
    \begin{equation}\label{eq:upper_bound_finite_bound}
        \e^{-\frac{x^2(t-k)}{2kt}} \int_0^\infty (y+1) \e^{\sqrt{2}y + xy/k - y^2/2k - y^2/2} \d{y} \lesssim 1.
    \end{equation}
    Hence~\eqref{eq:upper_bound_k}.
    Now, inserting~\eqref{eq:upper_bound_k} in~\eqref{eq:upper_bound_barrier}, we obtain
    \begin{equation*}
        \P(M_t > m_t + x) \lesssim \gamma_t(x) \sum_{k=0}^{t-1} \frac{t^{3/2}}{(t-k)^{3/2}(k+1)^{3/2}} \lesssim \gamma_t(x),
    \end{equation*}
    which concludes the proof.
\end{proof}

\begin{corollary}\label{cor:upper_bound_1}
    For any $x \in \R$ and any $t \geq 2$, $\P(M_t > m_t + x) \lesssim (1 + x_+) \e^{-\sqrt{2}x}$.
\end{corollary}

\begin{corollary}\label{cor:upper_bound_2}
    For any $t \geq 2$, any $\ell \in [2, t/\log t]$, and any $y \geq -\log \ell + 1$,
    \begin{equation*}
        \P \left( M_\ell > \frac{m_t}{t} \ell + y \right) \lesssim \frac{y + \log \ell}{\ell^{3/2}} \e^{-\sqrt{2}y - \frac{y^2}{2\ell} + \frac{3}{2\sqrt{2}} \frac{y \log t}{t}}.
    \end{equation*}
\end{corollary}

\begin{proof}
    It suffices to apply Lemma~\ref{lem:upper_bound} at time $\ell$ and with $x = y + \frac{3}{2\sqrt{2}}\log \ell - \frac{3}{2\sqrt{2}}\frac{\ell \log t}{t}$.
\end{proof}

\section{Non-contributing particles}\label{sct:non_contributing_particles}

From now on, we consider $(x_t)_{t \geq 0}$ such that $1 \ll x_t \ll t$ as $t \to \infty$.
Here and in the rest of the paper, it will be convenient to use the following straight barriers rather than $f_t(s)$ and $\tilde{f}_t(s)$,
\begin{equation}\label{eq:barriers}
    g_t(s) = \frac{m_t}{t}s + x_t \quad \text{and} \quad \tilde{g}_t(s) = g_t(s) - \sqrt{2}s = x_t - \frac{3}{2\sqrt{2}} \frac{\log t}{t} s.
\end{equation}
Let us introduce an arbitrary $(\ell_t)_{t \geq 0}$ such that $1 \ll \ell_t < t$ as $t \to \infty$.
We will gradually add upper bounds to $\ell_t$ as we progress through the paper.
What matters is only that this function diverges to infinity slowly enough.
To simplify notations, let us abbreviate $\ell = \ell_t$.
Define
\begin{equation*}
    G_t = \bigcup_{u \in \cN_{t-\ell}} \{ \exists s \leq t-\ell, X_s(u) > g_t(s), \max_{v \geq u} X_t(v) > m_t + x_t\},
\end{equation*}
where $v \geq u$ means that $v$ is a descendant of $u$.
The following proposition together with the upcoming Lemma~\ref{lem:lower_bound} tell that particles contributing to the upper moderate deviations typically stay below the barrier $g_t(s)$ until time $s = t-\ell$.

\begin{proposition}\label{prop:upper_bound_G}
    If $1 \ll \ell < t$, then $\P(G_t) \ll \gamma_t(x_t)$.
\end{proposition}

\begin{proof}
    To simplify notations, assume that $t-\ell$ is an integer.
    We have
    \begin{align}
        \P(G_t) &\leq \sum_{k=0}^{t-\ell-1} \E \left[ \sum_{u \in \cN_{k+1}} \mathds{1}_{\{\forall s \leq k, X_s(u) \leq g_t(s), \exists r \in [k,k+1], X_r(u) > g_t(r), \max_{v \geq u} X_t(v) > m_t + x_t\}} \right] \nonumber \\
        &\leq \sum_{k=0}^{t-\ell-1} \e^{k+1} u_t(k), \label{eq:upper_bound_G_barrier}
    \end{align}
    where
    \begin{equation*}
        u_t(k) = \P(\forall s \leq k, B_s \leq g_t(s), \exists r \in [k, k+1], B_r > g_t(r), M_{t-k-1} > m_t + x_t - B_{k+1}),
    \end{equation*}
    by the Markov property and Lemma~\ref{lem:many-to-one}.
    Consider $\varepsilon > 0$ and $\eta \in (0,1)$ to be determined later.
    We decompose the above sum into three, according if $k \in [0, \varepsilon x_t - 1)$, $k \in [\varepsilon x_t - 1, \eta t]$, or $k \in (\eta t, t - \ell - 1]$.\footnote{The argument still works if we replace the first threshold $\varepsilon x_t - 1$ with any $a_t$ such that $1 \ll a_t < \varepsilon x_t$.}
    
    First, if $k \in [0, \varepsilon x_t - 1)$, then we have the following rough bound
    \begin{equation*}
        u_t(k) \leq \P \left( \sup_{r \leq k+1} B_r > x_t \right) \lesssim \e^{-\frac{x_t^2}{2(k+1)}} \lesssim \e^{-\frac{x_t}{2 \varepsilon}}.
    \end{equation*}
    Hence,
    \begin{equation}\label{eq:upper_bound_G_1}
        \sum_{k=0}^{\varepsilon x_t - 1} \e^{k+1} u_t(k) \lesssim \varepsilon x_t \e^{\varepsilon x_t - \frac{x_t}{2\varepsilon}} \ll \gamma_t(x_t),
    \end{equation}
    as soon as $\varepsilon$ is small enough.

    Now, assume $k \in [\varepsilon x_t - 1, \eta t]$.
    Since the function~$s \mapsto g_t(s)$ is increasing, $u_t(k)$ is bounded by
    \begin{multline}\label{eq:upper_bound_G_2_decomposition}
        \P \left( \forall s \leq k, B_s \leq g_t(s), \sup_{r \in [k, k+1]} B_r > g_t(k), M_{t-k-1} > m_t + x_t - \sup_{r \in [k, k+1]} B_r \right) \\
        = \int_0^\infty \int_y^\infty \P(\forall s \leq k, B_s \leq g_t(s), B_k \in g_t(k) - \d{y}) \P(|B_1| \in \d{z}) \\
        \times \P \left( M_{t-k-1} > \frac{m_t}{t}(t-k) - z + y \right).
    \end{multline}
    We control the first factor in the above integral similarly to~\eqref{eq:upper_bound_Girsnov_and_ballot}, using the Girsanov transformation~\eqref{eq:girsanov_transformation} and the upper bound in Lemma~\ref{lem:brownian_estimate_1},
    \begin{equation}\label{eq:upper_bound_G_2_1}
        \P(\forall s \leq k, B_s \leq g_t(s), B_k \in g_t(k) - \d{y}) \lesssim \frac{x_t y}{k^{3/2}} \e^{-\sqrt{2}(\tilde{g}_t(k) - y) - k - (\tilde{g}_t(k) - y)^2/2k} \d{y}.
    \end{equation}
    For the second one, we use the explicit distribution of $|B_1|$ on $[0, \infty)$,
    \begin{equation}\label{eq:upper_bound_G_2_2}
         \P(|B_1| \in \d{z}) = \sqrt{\frac{2}{\pi}} \e^{-z^2/2} \d{z}.
    \end{equation}
    For the third one, we simply bound
    \begin{equation}\label{eq:upper_bound_G_2_3}
        \P \left( M_{t-k-1} > \frac{m_t}{t}(t-k) - z + y \right) \leq 1.
    \end{equation}
    We insert the three bounds~\eqref{eq:upper_bound_G_2_1}, \eqref{eq:upper_bound_G_2_2}, \eqref{eq:upper_bound_G_2_3} in~\eqref{eq:upper_bound_G_2_decomposition}, we expand the exponents, and we rewrite $-\frac{x_t^2}{2k} = -\frac{x_t^2}{2t} - \frac{x_t^2(t-k)}{2kt}$.
    Doing so, we obtain
    \begin{equation*}
        u_t(k) \lesssim \gamma_t(x_t) \frac{\e^{-k + \frac{3}{2}\frac{\log t}{t}k}}{k^{3/2}} \left( \e^{-\frac{x_t^2(t-k)}{2kt}} \int_0^\infty y \e^{\frac{x_t}{k}y - y^2/2k} \int_y^\infty \e^{\sqrt{2}z - z^2/2} \d{z} \d{y} \right).
    \end{equation*}
    By completing the square in the integral over $z$ and then applying~\eqref{eq:upper_bound_finite_bound}, we see that the factor in brackets is bounded by a constant.
    Moreover, since $\varepsilon x_t < k \leq \eta t$, we have $\frac{\log k}{k} \geq \frac{\log(\eta t)}{\eta t}$ and then $\frac{\log t}{t} k \leq (\eta + o(1)) \log k$, where $o(1)$ is a function that converges to $0$ as $t \to \infty$, uniformly in all the other variables.
    It follows that
    \begin{equation*}
        u_t(k) \lesssim \gamma_t(x_t) \frac{\e^{-k}}{k^{(1 - \eta - o(1))3/2}}.
    \end{equation*}
    Then,
    \begin{equation}\label{eq:upper_bound_G_2}
        \sum_{k = \varepsilon x_t - 1}^{\eta t} \e^{k+1} u_t(k) \lesssim \gamma_t(x_t) \sum_{k = \varepsilon x_t - 1}^\infty \frac{1}{k^{(1 - \eta - o(1))3/2}} \ll \gamma_t(x_t),
    \end{equation}
    as soon as $\eta + o(1) < 1/3$.

    Finally, the case $k \in (\eta t, t-\ell-1]$ can be treated in the same way as $k \in [\varepsilon x_t - 1, \eta t]$ but with a sharper control than~\eqref{eq:upper_bound_G_2_3}.
    Thanks to Corollary~\ref{cor:upper_bound_1}, we bound the left-hand side of~\eqref{eq:upper_bound_G_2_3} by
    \begin{equation*}
        \P \left( M_{t-k-1} > \frac{m_t}{t}(t-k-1) - z + y \right) \lesssim \log(t-k-1) \frac{\e^{\frac{3}{2}\frac{\log t}{t}(t-k-1)}}{(t-k-1)^{3/2}} \e^{\sqrt{2}(z-y)}.
    \end{equation*}
    Inserting this new bound, together with~\eqref{eq:upper_bound_G_2_1} and~\eqref{eq:upper_bound_G_2_2}, in~\eqref{eq:upper_bound_G_2_decomposition}, we obtain
    \begin{equation*}
        u_t(k) \lesssim \gamma_t(x_t) \e^{-k} \frac{\log(t-k-1) t^{3/2}}{k^{3/2}(t-k-1)^{3/2}} \left( \e^{-\frac{x_t^2(t-k)}{2kt}} \int_0^\infty y \e^{\frac{x_t}{k}y - y^2/2k} \int_y^\infty \e^{\sqrt{2}z - z^2/2} \d{z} \d{y} \right).
    \end{equation*}
    As in the case $k \in [\varepsilon x_t - 1, \eta t]$, the factor in brackets is bounded by a constant.
    It follows that
    \begin{equation}\label{eq:upper_bound_G_3}
        \sum_{k = \eta t}^{t-\ell} \e^{k+1} u_t(k) \lesssim \gamma_t(x_t) \sum_{k = \eta t}^{t-\ell} \frac{\log(t-k-1) t^{3/2}}{k^{3/2}(t-k-1)^{3/2}} \ll \gamma_t(x_t).
    \end{equation}
    
    By~\eqref{eq:upper_bound_G_barrier}, \eqref{eq:upper_bound_G_1}, \eqref{eq:upper_bound_G_2}, and~\eqref{eq:upper_bound_G_3}, we have $\P(G_t) \ll \gamma_t(x_t)$, which concludes the proof.
\end{proof}

\section{A first lower bound}

We still consider $(x_t)_{t \geq 0}$ such that $1 \ll x_t \ll t$ as $t \to \infty$.

\begin{lemma}\label{lem:lower_bound}
    There exists $t_0 > 0$ such that, for any $t \geq t_0$, $\P(M_t > m_t + x_t) \gtrsim \gamma_t(x_t)$.
\end{lemma}

\begin{proof}
    Here again, we will make use of the barriers $g_t(s)$ and $\tilde{g}_t(s)$ introduced in~\eqref{eq:barriers}.
    Define, for $u \in \cN_t$,
    \begin{equation*}
        H_t^{(u)} = \{\forall s \leq t, X_s(u) \leq g_t(s), X_t(u) > m_t + x_t - 1\} \quad \text{and} \quad \Delta_t = \sum_{u \in \cN_t} \mathds{1}_{H_t^{(u)}}.
    \end{equation*}
    By Cauchy-Schwarz,
    \begin{equation}\label{eq:lower_bound_cauchy-schwarz}
        \P(M_t > m_t + x_t - 1) \geq \frac{\E \left[ \mathds{1}_{\{M_t > m_t + x_t - 1\}} \Delta_t \right]^2}{\E \left[ \Delta_t^2 \right]} = \frac{\E \left[\Delta_t \right]^2}{\E \left[ \Delta_t^2 \right]}.
    \end{equation}
    Using successively Lemma~\ref{lem:many-to-one}, the Girsanov transformation~\eqref{eq:girsanov_transformation}, and the lower bound in Lemma~\ref{lem:brownian_estimate_1}, we obtain
    \begin{align}
        \E \left[ \Delta_t \right] &= \e^t \P(\forall s \leq t, B_s \leq g_t(s), B_t > g_t(t) - 1) \nonumber \\
        &\geq t^{3/2} \e^{-\sqrt{2}x_t} \P(\forall s \leq t, B_s \leq \tilde{g}_t(s), B_t > \tilde{g}_t(t) - 1) \nonumber \\
        &\gtrsim \gamma_t(x_t). \label{eq:lower_bound_expectation_Delta}
    \end{align}
    The study of $\E \left[ \Delta_t^2 \right]$ is more involved.
    We start by applying Lemma~\ref{lem:many-to-two},
    \begin{equation}\label{eq:lower_bound_many_to_two}
        \E \left[ \Delta_t^2 \right] = \E \left[ \Delta_t \right] + 2 \int_0^t \e^{2t-s} \P \left( \substack{\forall r \leq t, B_r^{1, s} \leq g_t(r), B_t^{1, s} > g_t(t) - 1 \\ \forall r \leq t, B_r^{2, s} \leq g_t(r), B_t^{2, s} > g_t(t) - 1} \right) \d{s},
    \end{equation}
    where $B^{1, s}$ and $B^{2, s}$ are two Brownian motions that coincide on $[0, s]$ and are independent afterward.
    By the Markov property,
    \begin{equation*}
        \P \left( \substack{\forall r \leq t, B_r^{1, s} \leq g_t(r), B_t^{1, s} > g_t(t) - 1 \\ \forall r \leq t, B_r^{2, s} \leq g_t(r), B_t^{2, s} > g_t(t) - 1} \right) = \E \left[ \mathds{1}_{\{\forall r \leq s, B_r \leq g_t(r)\}} \phi_{t,s}(B_s)^2 \right],
    \end{equation*}
    with
    \begin{align}
        \phi_{t,s}(z) &= \P(\forall r \leq t-s, B_r \leq g_t(r+s) - z, B_{t-s} > g_t(t) - z - 1) \nonumber \\
        &\lesssim t^{3/2} \e^{-\sqrt{2}x_t + \sqrt{2}z - t - s} \frac{g_t(s) - z}{(t-s)^{3/2}} \mathds{1}_{\{z \leq g_t(s)\}}, \label{eq:lower_bound_ballot_t-s}
    \end{align}
    where we have once again used the Girsanov transformation~\eqref{eq:girsanov_transformation} and the upper bound in Lemma~\ref{lem:brownian_estimate_1}.
    Then
    \begin{align*}
        \e^{2t - s} \P \left( \substack{\forall r \leq t, B_r^{1, s} \leq g_t(r), B_t^{1, s} > g_t(t) - 1 \\ \forall r \leq t, B_r^{2, s} \leq g_t(r), B_t^{2, s} > g_t(t) - 1} \right) &\lesssim \frac{t^3 \e^{-2\sqrt{2}x_t - 3s}}{(t-s)^3} \E \left[ \mathds{1}_{\{\forall r \leq s, B_r \leq g_t(r)\}} \e^{2\sqrt{2}B_s} (g_t(s) - B_s)^2 \right] \\
        &= \frac{t^3 \e^{-2\sqrt{2}x_t}}{(t-s)^3} \E \left[ \mathds{1}_{\{\forall r \leq s, B_r \leq \tilde{g}_t(r)\}} \e^{\sqrt{2}B_s} (\tilde{g}_t(s) - B_s)^2 \right].
    \end{align*}
    Decomposing over the values of $B_s$, we rewrite the last expectation as
    \begin{align}
        \int_0^\infty & \E \left[ \mathds{1}_{\{\forall r \leq s, B_r \leq \tilde{g}_t(r), B_s \in \tilde{g}_t(s) - \d{y}\}} \e^{\sqrt{2}B_s} (\tilde{g}_t(s) - B_s)^2 \right] \nonumber \\
        &= \int_0^\infty \e^{\sqrt{2}(\tilde{g}_t(s) - y)} y^2 \P(\forall r \leq s, B_r \leq \tilde{g}_t(r), B_s \in \tilde{g}_t(s) - \d{y}) \nonumber \\
        &\lesssim \frac{x_t}{s^{3/2}} \e^{\sqrt{2}x_t + \frac{3}{2\sqrt{2}}\frac{x_t \log t}{t} - \frac{3}{2}\frac{\log t}{t}s} \int_0^\infty y^3 \e^{-\sqrt{2}y - (x_t - y)^2/2s} \d{y}, \label{eq:lower_bound_ballot_s}
    \end{align}
    by Lemma~\ref{lem:brownian_estimate_1}.
    Note that the last integral is bounded by
    \begin{equation*}
        \int_0^\infty y^3 \e^{-\sqrt{2}y - (x_t - y)^2/2t} \d{y} \leq \e^{-x_t^2/2t} \int_0^\infty y^3 \e^{-(\sqrt{2} + o(1))y} \d{y} \lesssim \e^{-x_t^2/2t}.
    \end{equation*}
    It follows that
    \begin{equation*}
        \int_1^{t-1} \e^{2t-s} \P \left( \substack{\forall r \leq t, B_r^{1, s} \leq g_t(r), B_t^{1, s} > g_t(t) - 1 \\ \forall r \leq t, B_r^{2, s} \leq g_t(r), B_t^{2, s} > g_t(t) - 1} \right) \d{s} \lesssim \gamma_t(x_t) \int_1^{t-1} \frac{t^{3/2} \e^{\frac{3}{2} \frac{\log t}{t}(t-s)}}{(t-s)^3s^{3/2}} \d{s}.
    \end{equation*}
    Now, use that $\frac{\log t}{t}(t-s) \leq \log(t-s)$ for $t$ large enough and $s \leq t-1$ to bound
    \begin{equation*}
        \int_1^{t-1} \frac{t^{3/2} \e^{\frac{3}{2} \frac{\log t}{t}(t-s)}}{(t-s)^3s^{3/2}} \d{s} \leq \int_1^{t-1} \frac{t^{3/2}}{(t-s)^{3/2}s^{3/2}} \d{s} \leq 2 \int_1^\infty \frac{\d{s}}{s^{3/2}} = 4.
    \end{equation*}
    We treat the integration over $s \in [0, 1)$ similarly, except that we use the bound
    \begin{equation*}
        \P(\forall r \leq s, B_r \leq \tilde{g}_t(r), B_s \in \tilde{g}_t(s) - \d{y}) \leq \P(B_s \in \tilde{g}_t(s) - \d{y}) = \frac{\e^{-(\tilde{g}_t(s) - y)^2/2s}}{\sqrt{2 \pi s}} \d{y},
    \end{equation*}
    instead of Lemma~\ref{lem:brownian_estimate_1} in~\eqref{eq:lower_bound_ballot_s}.
    We treat the integration over $s \in (t-1, t]$ similarly as $s \in [1, t-1]$, except that we only use the Girsanov transformation~\eqref{eq:girsanov_transformation} to bound $\phi_{t,s}(z) \lesssim t^{3/2} \e^{-\sqrt{2}x_t + \sqrt{2}z - t - s}$ instead of~\eqref{eq:lower_bound_ballot_t-s}.
    Finally, we obtain
    \begin{equation*}
        \int_0^t \e^{2t-s} \P \left( \substack{\forall r \leq t, B_r^{1, s} \leq g_t(r), B_t^{1, s} > g_t(t) - 1 \\ \forall r \leq t, B_r^{2, s} \leq g_t(r), B_t^{2, s} > g_t(t) - 1} \right) \d{s} \lesssim \gamma_t(x_t).
    \end{equation*}
    Therefore, by~\eqref{eq:lower_bound_expectation_Delta} and~\eqref{eq:lower_bound_many_to_two}, we have $\E \left[ \Delta_t^2 \right] \lesssim \E \left[ \Delta_t \right]$.
    The lower bound~\eqref{eq:lower_bound_cauchy-schwarz} then becomes $\P(M_t > m_t + x_t - 1) \gtrsim \E \left[ \Delta_t \right]$.
    By~\eqref{eq:lower_bound_expectation_Delta}, this concludes the proof.
\end{proof}

\section{Reduction to an expectation}

In this section, we consider an arbitrary $\ell = \ell_t$ such that $1 \ll \ell < \min(t/\log t, x_t^5)$.
We stress the fact that, contrary to Section~\ref{sct:non_contributing_particles}, the bound $\ell < t$ will not be sufficient here.
We need $\ell < t/\log t$ to apply Corollary~\ref{cor:upper_bound_2} in the proof of Lemma~\ref{lem:equivalent_Lambda_Lambda2}.
The assumption $\ell < x_t^5$ is made to have an easy control of~\eqref{eq:equivalent_Lambda_Lambda2_2} below, but it is dispensable.
Define, for $u \in \cN_{t-\ell}$,
\begin{equation}\label{eq:definition_lambda}
    E_t^{(u)} = \{\forall s \leq t-\ell, X_s(u) \leq g_t(s), \max_{v \geq u} X_t(v) > m_t + x_t\} \quad \text{and} \quad \Lambda_t = \sum_{u \in \cN_{t-\ell}} \mathds{1}_{E_t^{(u)}}.
\end{equation}
Note that these objects depend on the choice of $(\ell_t)_{t \geq 0}$.
The main result of this section is the following equivalent.

\begin{proposition}\label{prop:expectation_equivalent}
    For any choice of $\ell$ such that $1 \ll \ell < \min(t/\log t, x_t^5)$, we have $\P(M_t > m_t + x_t) \sim \E \left[ \Lambda_t \right]$ as $t \to \infty$.
\end{proposition}

We will need two lemmas.

\begin{lemma}\label{lem:lower_bound_Lambda}
    There exists $t_0 > 0$ such that, for any $t \geq t_0$, $\E \left[ \Lambda_t \right] \gtrsim \gamma_t(x_t)$.
\end{lemma}

\begin{proof}
    We have $\E \left[ \Lambda_t \right] \geq \P(M_t > m_t + x_t) - \P(G_t)$.
    Then, it suffices to apply Lemma~\ref{lem:lower_bound} and Proposition~\ref{prop:upper_bound_G}.
\end{proof}

\begin{lemma}\label{lem:equivalent_Lambda_Lambda2}
    We have $\E \left[ \Lambda_t^2 \right] \sim \E \left[ \Lambda_t \right]$.
\end{lemma}

\begin{proof}
    Since
    \begin{equation*}
        \E \left[ \Lambda_t^2 \right] = \E \left[ \Lambda_t \right] + \E \left[ \sum_{u, v \in \cN_{t-\ell}, u \neq v} \mathds{1}_{E_t^{(u)}} \mathds{1}_{E_t^{(v)}} \right],
    \end{equation*}
    it suffices to show that the second term is negligible compared to $\E \left[ \Lambda_t \right]$.
    By Lemma~\ref{lem:many-to-two} and the Markov property,
    \begin{equation}\label{eq:equivalent_Lambda_Lambda2_many_to_two}
        \E \left[ \sum_{u, v \in \cN_{t-\ell}, u \neq v} \mathds{1}_{E_t^{(u)}} \mathds{1}_{E_t^{(v)}} \right] = 2 \int_0^{t-\ell} \e^{2(t-\ell)-s} \E \left[ \mathds{1}_{\{\forall r \leq s, B_r \leq g_t(r)\}} \psi_{t,s}(B_s)^2 \right] \d{s},
    \end{equation}
    with
    \begin{equation}\label{eq:equivalent_Lambda_Lambda2_definition_phi}
        \psi_{t,s}(z) = \P \left( \forall r \leq t-\ell-s, B_r \leq g_t(r+s) - z, M_\ell > m_t + x_t - B_{t-\ell-s} - z \right).
    \end{equation}
    
    We first deal with the integration over $s \in [\ell^{1/3}, t-\ell-\ell^{1/3}]$.
    Slight modifications of the argument will handle the regions $s \in [0, \ell^{1/3})$ and $s \in (t-\ell-\ell^{1/3}, t-\ell]$.
    Decomposing over the values of $B_s$, we rewrite
    \begin{equation}\label{eq:equivalent_Lambda_Lambda2_decomposition}
        \E \left[ \mathds{1}_{\{\forall r \leq s, B_r \leq g_t(r)\}} \psi_{t,s}(B_s)^2 \right] = \int_0^\infty \P \left( \forall r \leq s, B_r \leq g_t(r), B_s \in g_t(s) - \d{y} \right) \psi_{t,s} \left( g_t(s) - y \right)^2.
    \end{equation}
    Using the Girsanov transformation~\eqref{eq:girsanov_transformation} and Lemma~\ref{lem:brownian_estimate_1}, we obtain
    \begin{equation}\label{eq:equivalent_Lambda_Lambda2_decomposition_1}
        \P \left( \forall r \leq s, B_r \leq g_t(r), B_s \in g_t(s) - \d{y} \right) \lesssim \e^{-\sqrt{2} \left( \tilde{g}_t(s) - y \right) - s} \frac{x_t y \e^{- \left( \tilde{g}_t(s) - y \right)^2/2s}}{s^{3/2}} \d{y}.
    \end{equation}
    Besides, decomposing over the values of $B_{t-\ell-s}$ in the definition~\eqref{eq:equivalent_Lambda_Lambda2_definition_phi}, we rewrite $\psi_{t,s} \left( g_t(s) - y \right)$ as
    \begin{equation*}
        \int_0^\infty \P \left( \forall r \leq t-\ell-s, B_r \leq \frac{m_t}{t}r + y, B_{t-\ell-s} \in \frac{m_t}{t}(t-\ell-s) + y - \d{z} \right) \P \left( M_\ell > \frac{m_t}{t}\ell + z \right).
    \end{equation*}
    The first factor in the above integral can be controlled by using the Girsanov transformation~\eqref{eq:girsanov_transformation} and Lemma~\ref{lem:brownian_estimate_1},
    \begin{multline}\label{eq:equivalent_Lambda_Lambda2_decomposition_2}
        \P \left( \forall r \leq t-\ell-s, B_r \leq \frac{m_t}{t}r + y, B_{t-\ell-s} \in \frac{m_t}{t}(t-\ell-s) + y - \d{z} \right) \\
        \lesssim \e^{-\sqrt{2} \left( -\frac{3}{2\sqrt{2}} \frac{\log t}{t} (t-\ell-s) + y - z \right) - (t-\ell-s)} \frac{y z \e^{- \left( -\frac{3}{2\sqrt{2}} \frac{\log t}{t} (t-\ell-s) + y - z \right)^2/2(t-\ell-s)}}{(t-\ell-s)^{3/2}} \d{z}.
    \end{multline}
    As for the second factor, it is controlled by Corollary~\ref{cor:upper_bound_2},
    \begin{equation}\label{eq:equivalent_Lambda_Lambda2_decomposition_3}
        \P \left( M_\ell > \frac{m_t}{t}\ell + z \right) \lesssim \frac{z + \log \ell}{\ell^{3/2}} \e^{-\sqrt{2}z - \frac{z^2}{2\ell} + \frac{3}{2\sqrt{2}} \frac{z \log t}{t}}.
    \end{equation}
    Combining~\eqref{eq:equivalent_Lambda_Lambda2_decomposition_2} and~\eqref{eq:equivalent_Lambda_Lambda2_decomposition_3}, we obtain that $\psi_{t,s} \left( g_t(s) - y \right)$ is bounded, up to a multiplicative constant, by
    \begin{equation}\label{eq:equivalent_Lambda_Lambda2_decomposition_phi_1}
        \frac{y \e^{\frac{3}{2} \frac{\log t}{t} (t-\ell-s) - (\sqrt{2} - o(1))y - (t-\ell-s)}}{(t-\ell-s)^{3/2} \ell^{3/2}} \int_0^\infty z(z + \log \ell) \e^{-z^2/2\ell} \e^{-(y-z)^2/2(t-\ell-s)} \d{z},
    \end{equation}
    where we denote by $o(1)$ any function that converges to $0$ as $t \to \infty$, uniformly in all the other variables.
    We can bound the last exponential by $1$.
    It follows that
    \begin{equation}\label{eq:equivalent_Lambda_Lambda2_decomposition_phi_2}
        \psi_{t,s} \left( g_t(s) - y \right) \lesssim \frac{y \e^{\frac{3}{2} \frac{\log t}{t} (t-\ell-s) - (\sqrt{2} - o(1))y - (t-\ell-s)}}{(t-\ell-s)^{3/2}}.
    \end{equation}
    Then we insert~\eqref{eq:equivalent_Lambda_Lambda2_decomposition_1} and~\eqref{eq:equivalent_Lambda_Lambda2_decomposition_phi_2} into~\eqref{eq:equivalent_Lambda_Lambda2_decomposition}.
    Multiplying the resulting bound by $\e^{2(t-\ell)-s}$ and integrating over $[\ell^{1/3}, t-\ell-\ell^{1/3}]$ yields
    \begin{multline*}
        \int_{\ell^{1/3}}^{t-\ell-\ell^{1/3}} \e^{2(t-\ell)-s} \E \left[ \mathds{1}_{\{\forall r \leq s, B_r \leq g_t(r)\}} \psi_{t,s}(B_s)^2 \right] \d{s} \\
        \lesssim x_t \e^{-\sqrt{2}x_t} \int_{\ell^{1/3}}^{t-\ell-\ell^{1/3}} \frac{\e^{\frac{3}{2} \frac{\log t}{t} (t-\ell) + \frac{3}{2} \frac{\log t}{t} (t-\ell-s)}}{s^{3/2}(t-\ell-s)^3} \int_0^\infty y^3 \e^{-(\sqrt{2} - o(1))y - \left( \tilde{g}_t(s) - y \right)^2/2s} \d{y} \d{s}.
    \end{multline*}
    We can use that, for $t$ large enough and for $s \in [\ell^{1/3}, t-\ell-\ell^{1/3}]$,
    \begin{equation}\label{eq:equivalent_Lambda_Lambda2_log}
        \frac{\log t}{t} (t-\ell) + \frac{\log t}{t} (t-\ell-s) \leq \log (t-\ell)  +  \log (t-\ell-s),
    \end{equation}
    and
    \begin{equation*}
        \e^{- \left( \tilde{g}_t(s) - y \right)^2/2s} \leq \e^{- \left( \tilde{g}_t(s) - y \right)^2/2t} \leq \e^{-\frac{x_t^2}{2t} + \frac{3}{2\sqrt{2}} \frac{x_t \log t}{t} + o(1)y + o(1)}.
    \end{equation*}
    It follows that
    \begin{align}
        \int_{\ell^{1/3}}^{t-\ell-\ell^{1/3}} & \e^{2(t-\ell)-s} \E \left[ \mathds{1}_{\{\forall r \leq s, B_r \leq g_t(r)\}} \psi_{t,s}(B_s)^2 \right] \d{s} \nonumber \\
        &\lesssim \gamma_t(x_t) \int_{\ell^{1/3}}^{t-\ell-\ell^{1/3}} \frac{(t-\ell)^{3/2}}{s^{3/2}(t-\ell-s)^{3/2}} \int_0^\infty y^3 \e^{-(\sqrt{2} - o(1))y} \d{y} \d{s} \nonumber \\
        &\lesssim \gamma_t(x_t) \ell^{-1/6}. \label{eq:equivalent_Lambda_Lambda2_1}
    \end{align}
    
    We treat the region $s \in [0, \ell^{1/3})$ in the same way, except that we bound the left-hand side of~\eqref{eq:equivalent_Lambda_Lambda2_decomposition_1} by
    \begin{equation*}
        \P \left( B_s \in g_t(s) - \d{y} \right) = \e^{-\sqrt{2} \left( \tilde{g}_t(s) - y \right) - s} \frac{\e^{- \left( \tilde{g}_t(s) - y \right)^2/2s}}{\sqrt{2 \pi s}} \d{y}.
    \end{equation*}
    This yields
    \begin{align}
        \int_0^{\ell^{1/3}} \e^{2(t-\ell)-s} \E \left[ \mathds{1}_{\{\forall r \leq s, B_r \leq g_t(r)\}} \psi_{t,s}(B_s)^2 \right] \d{s} &\lesssim \frac{\gamma_t(x_t)}{x_t} \int_0^{\ell^{1/3}} \frac{(t-\ell)^{3/2}}{\sqrt{s} (t-\ell-s)^{3/2}} \d{s} \nonumber \\
        &\lesssim \gamma_t(x_t) \frac{\ell^{1/6}}{x_t}. \label{eq:equivalent_Lambda_Lambda2_2}
    \end{align}
    
    Concerning the region $s \in (t-\ell-\ell^{1/3}, t-\ell]$, we do the same as for $s \in [\ell^{1/3}, t-\ell-\ell^{1/3}]$ with three modifications.
    First, dropping the barrier, we bound the left-hand side of~\eqref{eq:equivalent_Lambda_Lambda2_decomposition_2}, up to a multiplicative constant, by
    \begin{equation*}
        \e^{-\sqrt{2} \left( -\frac{3}{2\sqrt{2}} \frac{\log t}{t} (t-\ell-s) + y - z \right) - (t-\ell-s)} \frac{\e^{- \left( -\frac{3}{2\sqrt{2}} \frac{\log t}{t} (t-\ell-s) + y - z \right)^2/2(t-\ell-s)}}{\sqrt{t-\ell-s}} \d{z}.
    \end{equation*}
    Then, in the counterpart of~\eqref{eq:equivalent_Lambda_Lambda2_decomposition_phi_1}, we keep $\e^{-(y-z)^2/2(t-\ell-s)}$ instead of $\e^{-z^2/2\ell}$.
    These first two modifications allow us to replace the inequality~\eqref{eq:equivalent_Lambda_Lambda2_decomposition_phi_2} with
    \begin{equation*}
        \psi_{t,s} \left( g_t(s) - y \right) \lesssim \frac{\e^{\frac{3}{2} \frac{\log t}{t} (t-\ell-s) - (\sqrt{2} - o(1))y - (t-\ell-s)}}{\ell^{3/2}}(\ell^{1/6} + y).
    \end{equation*}
    Finally, instead of~\eqref{eq:equivalent_Lambda_Lambda2_log}, we use that
    \begin{equation*}
        \frac{\log t}{t} (t-\ell) + \frac{\log t}{t} (t-\ell-s) \leq \log (t-\ell) + o(1).
    \end{equation*}
    These three modifications yield
    \begin{align}
        \int_{t-\ell-\ell^{1/3}}^{t-\ell} \e^{2(t-\ell)-s} \E \left[ \mathds{1}_{\{\forall r \leq s, B_r \leq g_t(r)\}} \psi_{t,s}(B_s)^2 \right] \d{s} &\lesssim \gamma_t(x_t) \int_{t-\ell-\ell^{1/3}}^{t-\ell} \frac{(t-\ell)^{3/2}}{s^{3/2}\ell^{8/3}} \d{s} \nonumber \\
        &\lesssim \gamma_t(x_t) \ell^{-7/3}. \label{eq:equivalent_Lambda_Lambda2_3}
    \end{align}
    
    By~\eqref{eq:equivalent_Lambda_Lambda2_1}, \eqref{eq:equivalent_Lambda_Lambda2_2}, \eqref{eq:equivalent_Lambda_Lambda2_3}, and Lemma~\ref{lem:lower_bound_Lambda}, the right-hand side of~\eqref{eq:equivalent_Lambda_Lambda2_many_to_two} is negligible compared to $\E \left[ \Lambda_t \right]$, which concludes the proof.
\end{proof}

\begin{remark}\label{rem:equivalent_Lambda_Lambda2}
    Since $\Lambda_t$ takes integer values, a consequence of Lemma~\ref{lem:equivalent_Lambda_Lambda2} is that
    \begin{equation*}
        \E \left[ \Lambda_t \mathds{1}_{\{\Lambda_t \geq 2\}} \right] \leq \E \left[ \Lambda_t^2 \right] - \E \left[ \Lambda_t \right] \ll \E \left[ \Lambda_t \right] = \P(\Lambda_t = 1) + \E \left[ \Lambda_t \mathds{1}_{\{\Lambda_t \geq 2\}} \right].
    \end{equation*}
    Therefore, $\E \left[ \Lambda_t \mathds{1}_{\{\Lambda_t \geq 2\}} \right] \ll \P(\Lambda_t = 1)$, which implies that $\P(\Lambda_t \geq 2) \ll \P(\Lambda_t = 1)$.
\end{remark}

\begin{proof}[Proof of Proposition~\ref{prop:expectation_equivalent}]
    First,
    \begin{equation}\label{eq:expectation_equivalent_upper_bound}
        \P(M_t > m_t + x_t) \leq \E \left[ \Lambda_t \right] + \P(G_t) \sim \E \left[ \Lambda_t \right],
    \end{equation}
    by Proposition~\ref{prop:upper_bound_G} and Lemma~\ref{lem:lower_bound_Lambda}.
    Besides, by applying successively Cauchy-Schwarz and Lemma~\ref{lem:equivalent_Lambda_Lambda2}, we obtain
    \begin{equation}\label{eq:expectation_equivalent_lower_bound}
        \P(M_t > m_t + x_t) \geq \frac{\E \left[ \Lambda_t \right]^2}{\E \left[ \Lambda_t^2 \right]} \sim \E \left[ \Lambda_t \right].
    \end{equation}
    Combining~\eqref{eq:expectation_equivalent_upper_bound} and~\eqref{eq:expectation_equivalent_lower_bound}, we deduce the equivalent in Proposition~\ref{prop:expectation_equivalent}.
\end{proof}

\section{Proof of the equivalent}

The main result of this section is the following equivalent.

\begin{proposition}\label{prop:equivalent_expectation}
    As soon as $1 \ll x_t \ll t$ and $1 \ll \ell \ll \min(t/\log t, x_t^5, t^2/x_t^2)$, we have $\E \left[ \Lambda_t \right] \sim C^* \gamma_t(x_t)$, where $C^*$ is the constant defined in Theorem~\ref{th:moderate_deviations}.
\end{proposition}

To prove Proposition~\ref{prop:equivalent_expectation}, we start by identifying the positions at time $t-\ell$ of the particles that mainly contribute to $\E \left[ \Lambda_t \right]$.
To this end, similarly to~\eqref{eq:definition_lambda}, define, for $A$ any Borel set of $[0, \infty)$ and $u \in \cN_{t-\ell}$,
\begin{equation*}
    E_{t,A}^{(u)} = E_t^{(u)} \cap \{X_{t-\ell}(u) \in g_t(t-\ell) - A\} \quad \text{and} \quad \Lambda_{t,A} = \sum_{u \in \cN_{t-\ell}} \mathds{1}_{E_{t,A}^{(u)}}.
\end{equation*}
Consider $a(s)$ and $b(s)$ any increasing functions such that $1 \ll a(\ell) \ll \sqrt{\ell} \ll b(\ell) \ll \min(\sqrt{t}, t/x_t)$.
Note that such functions exist as soon as the conditions of Proposition~\ref{prop:equivalent_expectation} are satisfied.
The following lemma shows that $\E \left[ \Lambda_t \right]$ is mainly supported by particles whose distance from the barrier $g_t(s)$ is of order $\sqrt{\ell}$ at time $s = t-\ell$.

\begin{lemma}\label{lem:equivalent_expectation}
    With the above choices of $\ell$, $a$, and $b$, taking $I(\ell)= [a(\ell), b(\ell)]$, we have
    \begin{equation}\label{eq:equivalent_expectation}
        \E \left[ \Lambda_t \right] \sim \E \left[ \Lambda_{t,I(\ell)} \right] \sim \sqrt{\frac{2}{\pi}} \gamma_t(x_t) \int_0^\infty y \e^{\sqrt{2}y} \P \left( M_\ell > \frac{m_t}{t}\ell + y \right) \d{y},
    \end{equation}
    and the equivalent still holds if we restrict the integration domain to $I(\ell)$.
\end{lemma}

\begin{proof}
    Using the Markov property, Lemma~\ref{lem:many-to-one}, and the Girsanov transformation~\eqref{eq:girsanov_transformation}, we rewrite $\E \left[ \Lambda_{t,A} \right]$ as
    \begin{equation}\label{eq:equivalent_expectation_decomposition}
        \int_A \e^{-\sqrt{2}(\tilde{g}_t(t-\ell) - y)} \P(\forall s \leq t-\ell, B_s \leq \tilde{g}_t(s), B_{t-\ell} \in \tilde{g}_t(t-\ell) - \d{y}) \P \left( M_\ell > \frac{m_t}{t}\ell + y \right),
    \end{equation}
    for any Borel set $A$.
    Then, to show the left-hand equivalent in~\eqref{eq:equivalent_expectation}, it suffices to show that~\eqref{eq:equivalent_expectation_decomposition} is negligible compared to $\E \left[ \Lambda_t \right]$ when $A = [0, \infty) \setminus [a(\ell), b(\ell)]$.
    By Lemma~\ref{lem:brownian_estimate_1},
    \begin{equation*}
        \P(\forall s \leq t-\ell, B_s \leq \tilde{g}_t(s), B_{t-\ell} \in \tilde{g}_t(t-\ell) - \d{y}) \lesssim \frac{x_t y}{(t-\ell)^{3/2}} \e^{-(\tilde{g}_t(t-\ell) - y)^2/2(t-\ell)} \d{y}.
    \end{equation*}
    Inserting this and Corollary~\ref{cor:upper_bound_2} into~\eqref{eq:equivalent_expectation_decomposition}, we obtain the new bound
    \begin{multline}\label{eq:equivalent_expectation_integral}
        \frac{x_t}{(t-\ell)^{3/2} \ell^{3/2}} \int_A \e^{-\sqrt{2}(\tilde{g}_t(t-\ell) - y)} y \e^{-(\tilde{g}_t(t-\ell) - y)^2/2(t-\ell)} (y + \log \ell) \e^{-\sqrt{2}y - \frac{y^2}{2\ell} + \frac{3}{2\sqrt{2}} \frac{y \log t}{t}} \d{y} \\
        \lesssim \gamma_t(x_t) \left( \ell^{-3/2} \int_A y(y + \log \ell) \e^{-y^2/2(t-\ell) + x_t y/(t-\ell) - y^2/2\ell} \d{y} \right),
    \end{multline}
    by definition of $\tilde{g}_t(t-\ell)$ and choice of $\ell$.
    But, since $a(\ell) \ll \sqrt{\ell} \ll t/x_t$,
    \begin{equation*}
        \ell^{-3/2} \int_0^{a(\ell)} y(y + \log \ell) \e^{-y^2/2(t-\ell) + x_t y/(t-\ell) - y^2/2\ell} \d{y} \lesssim \ell^{-3/2} \int_0^{a(\ell)} y(y + \log \ell) \d{y},
    \end{equation*}
    and this quantity vanishes as $t \to \infty$.
    Also,
    \begin{equation*}
        \ell^{-3/2} \int_{b(\ell)}^\infty y(y + \log \ell) \e^{-y^2/2(t-\ell) + x_t y/(t-\ell) - y^2/2\ell} \d{y} \lesssim \ell^{-3/2} \int_{b(\ell)}^\infty y^2 \e^{x_t y/(t-\ell) - y^2/2\ell} \d{y}.
    \end{equation*}
    With the change of variable $z = y/\sqrt{\ell}$, we see that this quantity vanishes as $t \to \infty$, by choice of $\ell$ and $b$.
    Besides, $\gamma_t(x_t) \lesssim \E \left[ \Lambda_t \right]$, by Lemma~\ref{lem:lower_bound_Lambda}.
    Subsequently, with $A = [0, \infty) \setminus [a(\ell), b(\ell)]$, we obtain that~\eqref{eq:equivalent_expectation_integral} is negligible with respect to $\E \left[ \Lambda_t \right]$, which yields the left-hand equivalent in~\eqref{eq:equivalent_expectation}.

    Let us show the right-hand equivalent in~\eqref{eq:equivalent_expectation}.
    By~\eqref{eq:brownian_estimate_0},
    \begin{equation*}
        \P \left( \forall s \leq t-\ell, B_s \leq \tilde{g}_t(s), B_{t-\ell} \in \tilde{g}_t(t-\ell) - \d{y} \right) = \left( 1 - \e^{-2 x_t y/(t-\ell)} \right) \frac{\e^{-(\tilde{g}_t(t-\ell) - y)^2/2(t-\ell)}}{\sqrt{2 \pi (t-\ell)}} \d{y},
    \end{equation*}
    on the domain $\{y > 0\}$.
    But $1 - \e^{-2 x_t y/(t-\ell)} \sim 2 x_t y/(t-\ell)$ as $t \to \infty$, and this equivalent is uniform in $y \in I(\ell)= [a(\ell), b(\ell)]$, by choice of $b$.
    Thus, in view of~\eqref{eq:equivalent_expectation_decomposition},
    \begin{align}
        \E \left[ \Lambda_{t,I(\ell)} \right] &\sim \int_{I(\ell)} \e^{-\sqrt{2}(\tilde{g}_t(t-\ell) - y)} \frac{2 x_t y}{(t-\ell)} \frac{\e^{-(\tilde{g}_t(t-\ell) - y)^2/2(t-\ell)}}{\sqrt{2 \pi (t-\ell)}} \P \left( M_\ell > \frac{m_t}{t}\ell + y \right) \d{y} \nonumber \\
        &= \sqrt{\frac{2}{\pi}} x_t \e^{-\sqrt{2}x_t - \frac{x_t^2}{2(t-\ell)} + \frac{3}{2\sqrt{2}}\frac{x_t \log t}{t}} \frac{\e^{\frac{3}{2}\frac{\log t}{t}(t-\ell)}}{(t-\ell)^{3/2}} \nonumber \\
        & \quad \times \int_{I(\ell)} y\e^{\sqrt{2}y} \e^{-\frac{9}{16}\frac{(\log t)^2}{t^2}(t-\ell) - \frac{y^2}{2(t-\ell)} + \frac{x_t y}{t-\ell} - \frac{3}{2\sqrt{2}}\frac{y \log t}{t}} \P \left( M_\ell > \frac{m_t}{t}\ell + y \right) \d{y} \nonumber \\
        &\sim \sqrt{\frac{2}{\pi}} \gamma_t(x_t) \int_{I(\ell)} y \e^{\sqrt{2}y} \P \left( M_\ell > \frac{m_t}{t}\ell + y \right) \d{y}, \label{eq:equivalent_expectation_integral_2}
    \end{align}
    since $\ell \ll \min(t/\log t, t^2/x_t^2)$ and $b(\ell) \ll \min(\sqrt{t}, t/x_t)$.
    Besides, similarly as above, Corollary~\ref{cor:upper_bound_2} implies that
    \begin{equation*}
        \sqrt{\frac{2}{\pi}} \gamma_t(x_t) \int_{[0, \infty) \setminus [a(\ell), b(\ell)]} y \e^{\sqrt{2}y} \P \left( M_\ell > \frac{m_t}{t}\ell + y \right) \d{y} \ll \gamma_t(x_t) \lesssim \E \left[ \Lambda_{t,I(\ell)} \right].
    \end{equation*}
    Thus, in~\eqref{eq:equivalent_expectation_integral_2}, we can replace the integration domain $I(\ell)$ with the whole $[0, \infty)$.
\end{proof}

\begin{proof}[Proof of Proposition~\ref{prop:equivalent_expectation}]
    By Proposition~\ref{prop:expectation_equivalent} and Lemma~\ref{lem:equivalent_expectation},
    \begin{equation*}
        \frac{\P(M_t > m_t + x_t)}{\gamma_t(x_t)} \sim \frac{\E \left[ \Lambda_t \right]}{\gamma_t(x_t)} \sim \sqrt{\frac{2}{\pi}} \int_{a(\ell)}^{b(\ell)} y \e^{\sqrt{2}y} \P \left( M_{\ell} > \frac{m_t}{t}\ell + y \right) \d{y},
    \end{equation*}
    for $a(s)$ and $b(s)$ any increasing functions such that $1 \ll a(\ell) \ll \sqrt{\ell} \ll b(\ell) \ll \min(\sqrt{t}, t/x_t)$.
    Applying successively the change of variable $z = y-\frac{3}{2\sqrt{2}}\frac{\log t}{t}\ell$ and Lemma~\ref{lem:equivalent_expectation}, we obtain that it is equivalent to
    \begin{equation*}
        \sqrt{\frac{2}{\pi}} \int_{a(\ell)-o(1)}^{b(\ell)-o(1)} z \e^{\sqrt{2}z} \P \left( M_{\ell} > \sqrt{2}\ell + z \right) \d{z} \sim \sqrt{\frac{2}{\pi}} \int_0^\infty z \e^{\sqrt{2}z} \P \left( M_{\ell} > \sqrt{2}\ell + z \right) \d{z}.
    \end{equation*}
    Let us denote by $f(\ell)$ this latter quantity.
    Since $\P(M_t > m_t + x_t)/\gamma_t(x_t)$ does not depend on the choice of $(\ell_t)_{t \geq 0}$, for any other $(\ell_t')_{t \geq 0}$ going to infinity slowly enough, we have $f(\ell_t) \sim f(\ell'_t)$ as $t \to \infty$.
    Then, an argument by contradiction shows that $f(s)$ converges to some constant $C^* \in (0,\infty)$ as $s \to \infty$.
    Hence, $\E \left[ \Lambda_t \right] \sim C^* \gamma_t(x_t)$.
\end{proof}

\begin{acks}
    I am grateful to Michel Pain for his help and careful reading.
    I also thank Bastien Mallein for his advice, which helped simplify the proof.
    Finally, I acknowledge the anonymous referees for their helpful comments and suggested corrections, which have improved the exposition of this paper.
\end{acks}

\bibliographystyle{amsplain}
\bibliography{biblio}

\end{document}